\documentclass[a4paper, 12pt]{article}
\usepackage{mathrsfs}
\usepackage{amsmath}
\usepackage{amsfonts}
\usepackage{amsthm}
\usepackage{authblk}
\usepackage{tcolorbox}
\usepackage{thmtools}
\usepackage{thm-restate}

\newtheorem{thm}{Theorem}[section]
\newtheorem{cor}[thm]{Corollary}
\newtheorem{lem}[thm]{Lemma}
\newtheorem{prop}[thm]{Proposition}

\newtheorem{pblm}[thm]{Problem}

\usepackage{enumerate} %
\ifx\pdfoutput\undefined
 \usepackage[dvipdfm,%
  pdfstartview=FitH, %
  bookmarks=true,%
  bookmarksnumbered=true, %
  bookmarksopen=true, %
  plainpages=false,%
  pdfpagelabels,%
  colorlinks=true, %
  linkcolor=blue, %
  citecolor=blue,%
  urlcolor=black]{hyperref}
  \AtBeginDvi{}  %
\else
 \usepackage[pdftex,%
  pdfstartview=FitH, %
  bookmarks=true,%
  bookmarksnumbered=true, %
  bookmarksopen=true, %
  plainpages=false,%
  pdfpagelabels,%
  colorlinks=true, %
  linkcolor=blue, %
  citecolor=blue,%
  urlcolor=black]{hyperref}%
\fi

\hypersetup{
    pdftitle={Equitable partition of planar graphs},
    pdfauthor={Ringi Kim, Sang-il Oum, and Xin Zhang}
}

\newcommand{\abs}[1]{\left\lvert #1 \right\rvert} %

\numberwithin{equation}{section}

\begin{document}

\title{Equitable partition of planar graphs}

\author[1]{Ringi Kim\thanks{
Supported by the National Research Foundation of Korea (NRF) grant funded by the Korea government (MSIT) (NRF-2018R1C1B6003786), and by INHA UNIVERSITY Research Grant.}}
\author[2,3]{Sang-il Oum\thanks{Supported by the Institute for Basic Science (IBS-R029-C1).}}
\author[4]{Xin Zhang\thanks{Supported by the National Natural Science Foundation of China (11871055) and the Youth Talent Support Plan of Xi'an Association for Science and Technology, China (2018-6).}}
\affil[1]{\small Department of Mathematics, Inha University, Incheon, Korea}
\affil[2]{\small Discrete Mathematics Group, 
Institute for Basic Science (IBS), Daejeon,~Korea}
\affil[3]{\small Department of Mathematical Sciences, KAIST, Daejeon,~Korea}
\affil[4]{\small School of Mathematics and Statistics, Xidian University, Xi'an~710071,~China}
\affil[ ]{\small Email: \texttt{ringikim@inha.ac.kr},
\texttt{sangil@ibs.re.kr}, 
\texttt{xzhang@xidian.edu.cn}}
\date\today

\maketitle

\begin{abstract}
    An \emph{equitable $k$-partition} of a graph $G$ is a collection of induced subgraphs $(G[V_1],G[V_2],\ldots,G[V_k])$ of $G$ such that $(V_1,V_2,\ldots,V_k)$ is a partition of $V(G)$
    and $-1\le |V_i|-|V_j|\le 1$ for all $1\le i<j\le k$.
    We prove that 
    every planar graph admits 
    an equitable $2$-partition into $3$-degenerate graphs,
    an equitable $3$-partition into 
    $2$-degenerate graphs, 
    and an equitable $3$-partition into
    two forests and one graph.


\end{abstract}

\section{Introduction}\label{sec:intro}

All graphs in this paper are simple and finite.
A \emph{$k$-partition} of a graph $G$ is 
a collection of 
induced subgraphs $(G[V_1],G[V_2],\ldots,G[V_k])$
such that 
$(V_1,V_2,\ldots,V_k)$ is a partition of $V(G)$. Such a $k$-partition is \emph{equitable} 
if 
\[\abs{\abs{V_i}-\abs{V_j}}\le 1\] for all $i, j\in\{1,2,\ldots,k\}$. %
If there is no confusion, then we  use $(V_1,V_2,\ldots,V_k)$ to denote a $k$-partition $(G[V_1],G[V_2],\ldots,G[V_k])$ of $G$.
We write $\Delta(G)$ to denote the maximum degree of a graph $G$. 

In 1970, Hajnal and Szemer\'edi~\cite{HS1970} proved
a conjecture of Erd\H{o}s, stating that every graph admits an equitable $k$-partition into empty subgraphs, if $k>\Delta(G)$. 
In 2008, Kierstead and Kostochka~\cite{KK2008} found a short proof.
In 2010, Kierstead, Kostochka, Mydlarz, and Szemer\'edi~\cite{KKMS2010} designed a fast algorithm to find such an equitable $k$-partition.
The bound on $k$ in the Hajnal-Szemer\'edi Theorem is sharp because of complete graphs for instance. Thus, there have been many results in this field trying to obtain better lower  bounds on the number $k$ of parts for special graph classes.
Motivated by Brooks' theorem, 
Chen, Lih, and Wu~\cite{CLW1994} conjectured that 
a connected graph~$G$ admits an equitable $\Delta(G)$-partition into empty graphs if and only if it is not 
$K_{\Delta(G)+1}$, an odd cycle, or $K_{\Delta(G),\Delta(G)}$ (for odd $\Delta(G)$).
They proved this conjecture for $\Delta(G)\le 3$ and Kierstead and Kostochka~\cite{KK2012} proved the conjecture for $\Delta(G)=4$.
For planar graphs, Zhang and Yap~\cite{ZhangYap1998} proved this conjecture for $\Delta(G)\ge 13$, and Nakprasit~\cite{Nakprasit20121019} 
 proved it for $\Delta(G)\ge 9$; in other words, he proved that every planar graph $G$  
 has an equitable $k$-partition into empty subgraphs if $k\ge \max(\Delta(G),9)$.

If we relax the condition on each part, then 
it is possible to reduce the number of parts significantly.
For instance, Williams, Vandenbussche, and Yu~\cite{WVY2012} proved that 
for all $k\ge 3$, every planar graph of minimum degree at least $2$ and girth at least $10$
has an equitable $k$-partition into graphs of maximum degree at most $1$.

We will mostly focus on the degeneracy of graphs.
A graph is \emph{$d$-degenerate}
if every non-null subgraph has a vertex of degree at most $d$. Note that a graph is $0$-degenerate if it has no edges, and $1$-degenerate if it is a forest.
Kostochka, Nakprasit, and Pemmaraju~\cite{KNP2005} studied the existence of an equitable $k$-partition
of a $d$-degenerate graph into $(d-1)$-degenerate graphs.
\begin{thm}[Kostochka, Nakprasit, and Pemmaraju~\cite{KNP2005}]\label{thm:degenerate}
    For $k\ge3$ and $d\ge2$,     
    every $d$-degenerate graph has an equitable $k$-partition into $(d-1)$-degenerate subgraphs.
\end{thm}
This implies that every $5$-degenerate graph admits 
an equitable $3$-partition into $4$-degenerate subgraphs,
an equitable $9$-partition into $3$-degenerate subgraphs,
an equitable $27$-partition into $2$-degenerate subgraphs,
and 
an equitable $81$-partition into forests.

Now we restrict our attention to planar graphs.
As planar graphs are $5$-degenerate, every planar graph  admits an equitable $81$-partition into forests. 
How far can we reduce 81? 
Esperet, Lemoine, and Maffray~\cite{ELM2015} proved 
that $81$ can be improved to $4$.
\begin{thm}[Esperet, Lemoine, and Maffray~\cite{ELM2015}]
    \label{thm:esperet}
    For all $k\ge 4$, every planar graph admits an equitable $k$-partition
    into forests.
\end{thm}
However it is not known whether $4$ is tight. Indeed,
Esperet, Lemoine, and Maffray~\cite{ELM2015} proposed the following problem: 
\begin{pblm}[Esperet, Lemoine, and Maffray~\cite{ELM2015}]\label{prbm:planar}
    Does every planar graph $G$ admit 
    an equitable $3$-partition into forests?
\end{pblm}

This problem still remains open and is known to have affirmative answers in the following cases:
\begin{itemize}
    \item 
    $G$ is $2$-degenerate, by Theorem~\ref{thm:degenerate}  (even if $G$ is non-planar),
    \item 
    the girth of $G$ is at least $5$, due to Wu, Zhang, and Li~\cite{WZL2013},
    \item 
    no two cycles of length at most $4$ share vertices in $G$, due to Zhang~\cite{Zhang2015}, 
    \item $G$ has no triangles, and 
    no two cycles of length $4$ are adjacent, due to Zhang~\cite{Zhang2015},
    \item $G$ has an acyclic $4$-coloring, due to Esperet, Lemoine, and Maffray~\cite{ELM2015}.
\end{itemize}

By relaxing the condition further, we may ask the following question. 
\begin{pblm}
    For each $i$, what is the minimum integer $k_i$ such that for all integers $k\ge k_i$, every planar graph admits an equitable $k$-partition into 
    $i$-degenerate subgraphs?
\end{pblm}
It is easy to see that $k_0=\infty$ by considering $K_{1,n}$ for large $n$, see Meyer~\cite{Meyer1973}.
Since every planar graph is $5$-degenerate, $k_i=1$ for all $i\ge 5$. 
Theorem~\ref{thm:esperet} implies that $k_1\le 4$.
Not every planar graph admits a (not necessarily equitable) $2$-partition into forests, shown by Chartrand and Kronk~\cite{CK1969}. 
Thus, $k_1\ge 3$.

Our first and second theorems prove that $k_3=k_4=2$ and $k_2\in\{2,3\}$.

\begin{restatable*}{thm}{threedeg}\label{two-3-degenerate}
    Every planar graph 
    admits an equitable $2$-partition
    into $3$-degenerate graphs.
\end{restatable*}

\begin{restatable*}{thm}{twodeg}\label{three-2-degenerate}
    Every planar graph 
    admits an equitable $3$-partition
    into $2$-degenerate graphs.
\end{restatable*}

Our third theorem shows a weaker variant of Problem~\ref{prbm:planar}.

\begin{restatable*}{thm}{twoforest}\label{thm:twoforest}
    Every planar graph 
    admits an equitable $3$-partition
    into two forests and one graph.
\end{restatable*}

The rest of this paper is organized as follows.
In Section~\ref{sec:2-degenerate}, we prove Theorems~\ref{two-3-degenerate} and \ref{three-2-degenerate}, and moreover, show that every triangle-free planar graph admits an equitable $2$-partition into $2$-degenerate graphs. In Section~\ref{sec:2F+1G}, we prove Theorem~\ref{thm:twoforest} and illustrate some discussions towards Problem~\ref{prbm:planar} and its relative problems.

\section{Equitable partition into degenerate graphs} \label{sec:2-degenerate}
For a graph $G$ and disjoint sets $U$, $V$ of vertices of $G$, we denote by $e_G(U,V)$ the number of edges between $U$ and $V$. 
If $U=\{u\}$ or $V=\{v\}$, then we simply write $e_G(u,V)$ or $e_G(U,v)$ for $e_G(U,V)$.
For a vertex set $S\subseteq V(G)$ and vertices $v \in S$ and $u \notin V(G)- S$, let us write $S-v$ for the set $S-\{v\}$ and $S+u$ for the set $S\cup\{u\}$.

Our first theorem shows that $k_3=k_4=2$.
\threedeg

\begin{proof}
Let $G$ be an $n$-vertex planar graph.
We proceed by induction on  $\abs{E(G)}$. 
We may assume that $G$ has at least one edge and $n\ge 4$.

As $G$ is planar, 
it has a vertex $v$ 
such that $0<\deg(v)\le5$.
Let $v_1$ be a neighbor of $v$. By the induction hypothesis, 
there is an equitable $2$-partition $(V_1,V_2)$ of  $G-vv_1$ into  $3$-degenerate graphs. 
We may assume, without loss of generality, that $v\in V_1$.  
If $e_G(v,V_1-v)\leq 3$, then $(V_1,V_2)$ is an equitable $2$-partition of $G$ into $3$-degenerate graphs. %
So we may assume that 
$e_G(v,V_1-v)\geq 4$, and so $e_G(v,V_2)\leq 1$. Therefore,
$V_2+v$ induces a $3$-degenerate subgraph of $G$.

If there is a vertex $w\in V_2$ so that $e_G(w,V_1-v)\leq 3$, then $(V_1-v+w,V_2-w+v)$ is an equitable $2$-partition  of $G$ into $3$-degenerate graphs. %
Hence we assume that $e_G(w,V_1-v)\geq 4$ for every $w\in V_2$, which implies that
\[
    e_G(V_2,V_1-v)\geq 4|V_2| \ge 4 \lfloor n/2\rfloor \ge 2n-2.%
\]
On the other hand, the graph induced by the edges between $V_1-v$ and $V_2$ is a bipartite planar graph on $n-1$ vertices, and therefore 
$ e_G(V_2,V_1-v)\leq 2(n-1)-4=2n-6$, contradicting the other inequality.
\end{proof}

Now we show that $2\le k_2\le 3$.
\twodeg

\begin{proof}
Let $G$ be an $n$-vertex planar graph. %
We proceed by induction on $\abs{E(G)}$.
We may assume that $G$ has at least one edge and at least $4$ vertices.

Since $G$ is planar, there is a vertex $v$ such that 
$1\le \deg(v)\le 5$.
Let $v_1$ be a neighbor of $v$. 
By applying the induction hypothesis to the graph $G-vv_1$, we obtain an equitable $3$-partition $(V_1,V_2,V_3)$ of 
$G-vv_1$ into  $2$-degenerate graphs. %
We may assume, without loss of generality, that $v\in V_1$.  
If $e_G(v,V_1-v)\leq 2$,
then $(V_1,V_2,V_3)$ is also an equitable $3$-partition of $G$ into $2$-degenerate graphs. %
So we may assume that  $e_G(v,V_1-v)\geq 3$, which implies that $e_G(v,V_2)\leq 2$ and $e_G(v,V_3)\leq 2$. Therefore,
both $V_2+ v$ and $V_3+v$ induce $2$-degenerate subgraphs of $G$.

If there is a vertex $w\in V_2$ so that $e_G(w,V_1-v)\leq 2$, then $V_1-v+w$ induces a 2-degenerate subgraph of $G$. 
Hence, $(V_1-v+w, V_2-w+v, V_3)$ is an equitable $3$-partition of $G$ into $2$-degenerate graphs.
Now we assume that $e_G(w,V_1-v)\geq 3$ for every $w\in V_2$, and by symmetry, we assume further that $e_G(w,V_1-v)\geq 3$ for every $w\in V_3$.
This implies that
\[
    e_G(V_2\cup V_3,V_1-v)\geq 3|V_2\cup V_3|\ge 3 \cdot 2\lfloor n/3  \rfloor
    \ge 2(n-2).
\]

On the other hand, the graph induced by the edges between $V_2\cup V_3$ and $V_1-v$ is a bipartite planar graph on $n-1$ vertices, 
and therefore 
\[
e_G(V_2\cup V_3,V_1-v)\leq 2(n-1)-4,
\]
contradicting the previous inequality.
\end{proof}

We do not know whether $k_2=2$. %

\begin{pblm}\label{prb:22partition}
Does every planar graph admit an equitable $2$-partition into $2$-degenerate graphs?
\end{pblm}

We remark that Thomassen~\cite{Thomassen1995} proved that 
every planar graph admits a $2$-partition into 
a $2$-degenerate graph and a forest.

The following theorem shows that a possible counterexample to Problem~\ref{prb:22partition} shall contain triangles.  

\begin{thm}\label{two-2-degenerate-girth}
    Every triangle-free planar graph
    admits an equitable $2$-partition into $2$-degenerate graphs.
\end{thm}

\begin{proof}

We first prove by using the  discharging method that every triangle-free planar graph contains a vertex of degree at most $2$ or an edge with one end having degree $3$ and the other having degree at most $6$. 

Suppose there exists a triangle-free planar graph $G$ not satisfying the condition above. 
We assign initial charge $\mu(v)=\deg(v)$ to each vertex~$v$ of~$G$, 
and let each vertex of degree at least $7$ send $1/3$ to each of its neighbors of degree $3$. 
After this discharging, the final charge $\mu'(v)$ of each vertex $v$ of degree $3$ is exactly $3+3\times 1/3=4$ 
since vertices of degree $3$ are only adjacent  to vertices of degree at least $7$. 
If $v$ is a vertex of degree in $\{4,5,6\}$, then $\mu'(v)=\mu(v)\geq 4$. Lastly, each vertex $v$ of degree at least $7$ has final charge $\mu'(v)\ge \deg(v)-\deg(v)/3\geq 14/3>4$. Since $G$ is triangle-free and has minimum degree at least $3$,
\[
    4n>2e(G)=\sum_{v\in V(G)}\deg(v)=\sum_{v\in V(G)}\mu(v)=\sum_{v\in V(G)}\mu'(v)\ge 4n,
\]
 a contradiction. Therefore, the claim holds.

Now we prove the theorem.

Let $G$ be an $n$-vertex triangle-free planar graph. We proceed by induction on $n$. 
If there is a vertex $v$ of degree at most two, then by applying the induction hypothesis to the graph $G-v$, we obtain an equitable $2$-partition $(V_1,V_2)$ of 
$G-v$ into $2$-degenerate graphs with $|V_1|\le |V_2|$. 
Then, $(V_1+v, V_2)$ is an equitable $2$-partition of $G$ into $2$-degenerate graphs. %

So we assume that every vertex of $G$ has degree at least $3$.
Then, by the claim, there is an edge $uv$ with $\deg(u)=3$ and $\deg(v)\leq 6$. 
Applying the induction hypothesis to the graph $G-\{u,v\}$, we obtain an equitable $2$-partition $(V_1,V_2)$ of $G-\{u,v\}$ into two 2-degenerate graphs. %
Since $v$ has at most five neighbors in $G-\{u,v\}$, 
we may assume that $V_1$  contains at most two neighbors of $v$.
Then $(V_1+v,V_2+u)$ is an equitable $2$-partition of $G$ into $2$-degenerate graphs.
This completes the proof.
\end{proof}

\section{Equitable partition into $2$ forests and $1$ graph}\label{sec:2F+1G}
In this section, we aim to prove the following theorem.

\twoforest

An \emph{acyclic $k$-coloring} of a graph is a proper $k$-coloring
such that there is no cycle consisting of two colors.
In other words, if a graph has an acyclic $k$-coloring, then 
its vertex set can be partitioned into $k$ independent sets $A_1$, $A_2$, $\ldots$, $A_k$
such that $A_i\cup A_j$ induces a forest for all 
$i, j\in\{1,2,\ldots,k\}$. %
Borodin proved the following theorem, initially conjectured by Gr\"unbaum~\cite{Grunbaum1973}.
\begin{thm}[Borodin~\cite{Borodin1976}]\label{thm:borodin}
    Every planar graph has an  acyclic $5$-coloring.
\end{thm}

To prove Theorem~\ref{thm:esperet}, 
Esperet, Lemoine, and Maffray used Theorem~\ref{thm:borodin}
and try to combine two color classes in an acyclic $5$-coloring of planar graphs
to produce large induced forests.
We extend their idea.

\subsection{Key proposition}\label{subsec:prop}
To prove Theorem~\ref{thm:twoforest}, we prove the following stronger statement. 

\begin{prop}\label{prop:main}
    Let $k> \ell\ge 1$ be integers.
    Let $A_1$, $A_2$, $\ldots$, $A_k$ be sets
    such that $\abs{\bigcup_{i=1}^k A_i}=n$.
    Let \[q= \left\lfloor \frac{2}{k+\ell-1}\left(n+\frac{k-\ell}{2}\right)\right\rfloor.\]
    Then there exists 
    a partition $(B_0,B_1,\ldots,B_\ell)$ of $\bigcup_{i=1}^k A_i$ 
    into $\ell+1$  sets, possibly empty,
    such that
    \begin{enumerate}[(i)]
    \item    for each $1\le i\le \ell$, 
    $B_i$ is a subset of the union of two members of $\{A_1,A_2,\ldots,A_k\}$,
    \item $\abs{B_i}\ge q+1$ if $1\le i\le \lceil n- \frac{k+\ell-1}{2}q\rceil$,
    \item $\abs{B_i}\ge q$ if $\lceil n- \frac{k+\ell-1}{2}q\rceil< i\le \ell$,
    \item there exists $I\subseteq \{1,2,\ldots,k\}$ with $\abs{I}=k-\ell-1$
    such that $B_0\subseteq \bigcup_{i\in I} A_i$.
    \end{enumerate} 
\end{prop}
Let us first see why Proposition~\ref{prop:main} 
together with Theorem~\ref{thm:borodin}
implies Theorem~\ref{thm:twoforest}
\begin{proof}[Proof of Theorem~\ref{thm:twoforest}]
    Let $G$ be a planar graph with $n$ vertices.
    Then $G$ has an acyclic $5$-coloring by Theorem~\ref{thm:borodin} and so 
    there exist sets $A_1$, $A_2$, $\ldots$, $A_5$ such that $A_1\cup A_2\cup A_3\cup A_4\cup A_5=V(G)$
    and $A_i\cup A_j$ induces a forest for all $1\le i<j\le 5$. 
    By applying Proposition~\ref{prop:main} with $k:=5$ and $\ell:=2$, we have a partition $(B_0,B_1,B_2)$ of $V(G)$ such that 
    $\abs{B_1},\abs{B_2}\ge q$ 
    and 
    both $G[B_1]$ and $G[B_2]$ are forests, 
    where $q=\lfloor \frac{2}{6}(n+\frac{3}{2})\rfloor
    =\lfloor \frac{n+1}{3}\rfloor$.
    If $n\not\equiv2\pmod3$, then we take $B_1'\subseteq B_1$ and $B_2'\subseteq B_2$ such that $\abs{B_1'}=\abs{B_2'}=\lfloor n/3\rfloor$. 
    If $n\equiv2\pmod3$, then we take $B_1'\subseteq B_1$ and $B_2'\subseteq B_2$ such that $\abs{B_1'}=\abs{B_2'}=\lfloor n/3\rfloor+1$. Let $B_0'=V(G)\setminus (B_1'\cup B_2')$. 
    Then $(B_0',B_1',B_2')$ is a desired equitable $3$-partition.
\end{proof}

Here is a key lemma to prove Proposition~\ref{prop:main} inductively.

\begin{lem}\label{lem:number}
    Let $2\le \ell<k$ and $n$ be positive integers and let
    $q=\lfloor \frac{2n+k-\ell}{k+\ell-1}\rfloor$ and $n'=n-q$.
    Then 
    \begin{align*}
        \left\lfloor \frac{2n'+k-\ell}{k+\ell-3}
        \right\rfloor
        &=
        \begin{cases}
        q+1
        &\text{if }\lceil n- \frac{k+\ell-1}{2}q\rceil\ge \ell-1,
        \\
        q
        &\text{otherwise,}
        \end{cases}
    \\
    \intertext{and}
    n-\frac{k+\ell-1}{2}q &= n'-\frac{k+\ell-3}{2}q.
   \end{align*}
\end{lem}
\begin{proof}
Let $m:=k+\ell-1$, and let $r$ be the integer
such that $2n+k-\ell=mq+r$ and $0\le r< m$. 
Then
\[
\left\lfloor \frac{2n'+k-\ell}{k+\ell-3}
       \right\rfloor = \left\lfloor \frac{2(n-q)+k-\ell}{m-2}
        \right\rfloor
    = \left\lfloor \frac{(m-2)q+r}{m-2}\right\rfloor
    = q+\left\lfloor \frac{r}{m-2}\right\rfloor.
\]
Note that $\lceil n-mq/2\rceil \ge \ell-1$
if and only if $r\ge k+\ell-3=m-2$. Thus the first equation holds.
The second equation is trivial.
This completes the proof.
\end{proof}

\begin{proof}[Proof of Proposition~\ref{prop:main}]
    We first observe that $\lceil n- \frac{k+\ell-1}{2}q\rceil\le \ell-1$, since 
    \[
    \frac{2n-2\ell+2}{k+\ell-1} \le \left\lfloor \frac{(2n-2\ell+2)+(k+\ell-2)}{k+\ell-1}\right\rfloor =\left\lfloor\frac{2n+k-\ell}{k+\ell-1}\right\rfloor=q.
    \]
    We may assume that $A_i\cap A_j=\emptyset$ for all $i\neq j$.
    We proceed by induction on~$\ell$.
    Let $a_i=\abs{A_i}$ for all $1\le i\le k$.
    
    If $\ell=1$, then 
    by the pigeonhole principle, there exist $1\le i<j\le k$
    such that $a_i+a_j\ge \lceil \frac{2n}{k}\rceil$.
    Observe that \[q=\left\lfloor 
    \frac{2}{k} \left( n+\frac{k-1}{2}\right)\right\rfloor
    = \left\lfloor \frac{2n}{k}+\frac{k-1}{k}\right\rfloor 
    = \left\lceil \frac{2n}{k}\right\rceil
    \]
    and therefore we take  a set $B_1= A_i\cup A_j$.
    Then $\abs{B_1}\ge q$
    and $B_0$ is the union of the $k-2$ members of $\{A_1,A_2,\ldots,A_k\}\setminus \{A_i,A_j\}$.
    Now we may assume that $\ell>1$.
    
    Suppose that there exist $i\neq j$ such that
    $a_i\le q \le a_i+a_j$.
    Without loss of generality, let us assume $i=1$ and $j=2$.
    Then there exists $B_\ell$ such that 
    $A_1\subseteq B_\ell\subseteq A_1\cup A_2$
    and $\abs{B_\ell}=q$.
    Let \[n'=n-q\quad\text{ and }\quad q'=\left\lfloor \frac{2}{k+\ell-3} \left( n'+\frac{k-\ell}{2}\right)\right\rfloor.\]
    By applying the induction hypothesis to the $k-1$ sets $A_2-B_\ell, A_3, A_4,\ldots, A_k$, we obtain a partition 
    $(B_0,B_1,B_2, \ldots, B_{\ell-1})$ of $\bigcup_{i=2}^k A_i - B_\ell$,
    such that 
    for $1\le i\le \ell-1$, 
    the set $B_i$ is a subset of the union of two members of $\{A_2,A_3,\ldots,A_k\}$
    and  $B_0$ is a subset of the union of $k-\ell-1$ members of  $\{A_2,A_3,\ldots,A_k\}$.
    If $\lceil n-\frac{k+\ell-1}{2}q\rceil = \ell-1$, then by Lemma~\ref{lem:number},
    $q'=q+1$ and therefore the induction hypothesis provides that 
    $\abs{B_i}\ge q+1$ for all $1\le i\le \ell-1$.
    If $\lceil n-\frac{k+\ell-1}{2}q\rceil< \ell-1$, 
    then again by Lemma~\ref{lem:number}, $q'=q$ and $n-\frac{k+\ell-1}{2}q=n'-\frac{k+\ell-3}{2}q'$  and therefore 
    we deduce (ii) and (iii) by the induction hypothesis.

    Thus, we may assume that for all $i\neq j$, 
     $a_i>q$ or $a_i+a_j<q$.
    Note that if $a_i>q$, 
    then     $a_i+a_j     >q$ and therefore $a_j>q$.
    Thus we deduce that either 
    \begin{enumerate}[(I)]
    \item     $a_i> q$
    for all $1\le i\le k$,
    or 
    \item 
    $   a_i+a_j< q$     for all $i\neq j$. 
    \end{enumerate}
    By the pigeonhole principle there exist $1\le i<j\le k$
    such that $a_i+a_j\ge \lceil \frac{2n}{k}\rceil$.
    Since 
    \[
    \left\lceil \frac{2n}{k}\right\rceil
    =
    \left\lfloor\frac{2}{k}\left( n+ \frac{k-1}{2}\right)\right\rfloor
    \ge 
    \left\lfloor \frac{2}{k+\ell-1} \left(n+\frac{k-\ell}{2}\right)\right\rfloor=q,
    \]
 (II) does not hold and so (I) holds. Then we take $B_i=A_i$ for $i=1,2,\ldots,\ell-1$, $B_\ell=A_\ell\cup A_{\ell+1}$,
 and $B_0=\bigcup_{i=\ell+2}^k A_i$.
\end{proof}

Proposition~\ref{prop:main} is best possible in the sense that we cannot increase $q$;
consider the case that $A_1, A_2,\ldots, A_k$ are disjoint sets with  $\abs{A_1}=\ell a$,
$\abs{A_2}=\abs{A_3}=\cdots=\abs{A_k}=a$ for some positive integer $a$.
There are no $\ell$ disjoint sets of size at least $2a+1$, each contained in the  union of two members of $\{A_1,\ldots,A_k\}$, so we cannot increase $q$ from $2a$ to $2a+1$.

When $\ell=k-1$, then we obtain the following from Proposition~\ref{prop:main}, 
which is due to Esperet, Lemoine, and Maffray~\cite{ELM2015}.
\begin{cor}[Esperet, Lemoine, and Maffray~\cite{ELM2015}]
    \label{cor:k-1}
    Let $k>1$ be an integer. 
    Let $A_1$, $A_2$, $\ldots$, $A_k$ be sets
    such that $\abs{\bigcup_{i=1}^k A_i}=n$.
    Then there exists 
    a partition $(B_1,\ldots,B_{k-1})$ of $\bigcup_{i=1}^k A_i$ 
    into $k-1$  sets
    such that
    for each $1\le i\le k-1$, 
    $B_i$ is a subset of the union of two members of $\{A_1,A_2,\ldots,A_k\}$,
    and $\abs{B_i}=\lceil \frac{n}{k-1}\rceil$ or 
    $\abs{B_i}=\lfloor \frac{n}{k-1}\rfloor$ 
\end{cor}
\begin{proof}
    We apply Proposition~\ref{prop:main} with $\ell=k-1$
    to obtain $(B_0,B_1,\ldots,B_{k-1})$. Then $q=\lfloor \frac{1}{k-1}(n+\frac{1}{2})\rfloor=\lfloor \frac{n}{k-1}\rfloor$. 
    As $B_0$ is a subset of the union of $0$ sets, $B_0=\emptyset$.
    And, $\lceil n-\frac{k+\ell-1}{2}q\rceil
    = n-(k-1)q$ is exactly 
    the remainder when dividing $n$ by $k-1$ and therefore
    $\abs{B_i}=q+1$ for all $1\le i\le \lceil n-\frac{k+\ell-1}{2}q\rceil$
    and $\abs{B_i}=q$ for all $\lceil n-\frac{k+\ell-1}{2}q\rceil<i\le k-1$.
    Thus $(B_1,B_2,\ldots,B_{k-1})$ is a desired partition.
\end{proof}

\subsection{Discussions}\label{subsec:disc}

Borodin and Ivanova~\cite{BI2013} and Chen and Raspaud~\cite{CR2013} independently showed that every planar graph without cycles of length $4$ or $5$
has an acyclic $4$-coloring.
By Corollary~\ref{cor:k-1}, 
if a planar graph has 
    no cycle of length $4$ or $5$, then it 
    admits an equitable $3$-partition into forests.
By Proposition~\ref{prop:main}, we have the following variation of Problem~\ref{prbm:planar}.
\begin{cor}
    If a planar graph has 
    no cycle of length $4$ or $5$, 
    then it admits a partition of its vertex set
    into three sets $A_1$, $A_2$, $A_3$
    such that
    \begin{enumerate}[(i)]
     \item  both $A_1$ and $A_2$ induce forests and    $\abs{A_1}, \abs{A_2}\ge \lfloor \frac{2}{5}(n+1)\rfloor$, 
     \item $A_3$ is independent.
    \end{enumerate}
 \end{cor}
By the four color theorem, 
Corollary~\ref{cor:k-1} implies that every planar graph 
    admits an equitable $3$-partition
    into bipartite graphs.
We also deduce the following variant.

\begin{cor}
   Every $n$-vertex planar graph 
   admits a partition of its vertex set
   into three sets $A_1$, $A_2$, $A_3$
   such that
   \begin{enumerate}[(i)]
    \item  both $A_1$ and $A_2$ induce bipartite subgraphs and    $\abs{A_1}, \abs{A_2}\ge \lfloor \frac{2}{5}(n+1)\rfloor$, 
    \item $A_3$ is independent.
   \end{enumerate}
\end{cor}

A \emph{linear forest} is a forest of maximum degree at most $2$.
Cai, Xie, and Yang~\cite{CXY2012} showed that for every planar graph $G$, its vertices
can be colored by $\Delta(G)+7$ colors such that 
the union of any two color classes induces a linear forest.
Combined with Proposition~\ref{prop:main} and Corollary~\ref{cor:k-1}, we deduce that every planar graph $G$ 
admits an equitable $(\Delta(G)+6)$-partition 
into linear forests.


\begin{thebibliography}{10}

    \bibitem{Borodin1976}
    O.~V. Borodin.
    \newblock A proof of {B}. {G}r\"{u}nbaum's conjecture on the acyclic
      {$5$}-colorability of planar graphs.
    \newblock {\em Dokl. Akad. Nauk SSSR}, 231(1):18--20, 1976.
    
    \bibitem{BI2013}
    O.~V. Borodin and A.~O. Ivanova.
    \newblock Acyclic 4-choosability of planar graphs with no 4- and 5-cycles.
    \newblock {\em J. Graph Theory}, 72(4):374--397, 2013.
    
    \bibitem{CXY2012}
    C.~Cai, D.~Xie, and W.~Yang.
    \newblock A result on linear coloring of planar graphs.
    \newblock {\em Inform. Process. Lett.}, 112(22):880--884, 2012.
    
    \bibitem{CK1969}
    G.~Chartrand and H.~V. Kronk.
    \newblock The point-arboricity of planar graphs.
    \newblock {\em J. London Math. Soc.}, 44:612--616, 1969.
    
    \bibitem{CLW1994}
    B.~L. Chen, K.-W. Lih, and P.-L. Wu.
    \newblock Equitable coloring and the maximum degree.
    \newblock {\em European J. Combin.}, 15(5):443--447, 1994.
    
    \bibitem{CR2013}
    M.~Chen and A.~Raspaud.
    \newblock Planar graphs without 4- and 5-cycles are acyclically 4-choosable.
    \newblock {\em Discrete Appl. Math.}, 161(7-8):921--931, 2013.
    
    \bibitem{ELM2015}
    L.~Esperet, L.~Lemoine, and F.~Maffray.
    \newblock Equitable partition of graphs into induced forests.
    \newblock {\em Discrete Math.}, 338(8):1481--1483, 2015.
    
    \bibitem{Grunbaum1973}
    B.~Gr\"{u}nbaum.
    \newblock Acyclic colorings of planar graphs.
    \newblock {\em Israel J. Math.}, 14:390--408, 1973.
    
    \bibitem{HS1970}
    A.~Hajnal and E.~Szemer\'{e}di.
    \newblock Proof of a conjecture of {P}. {E}rd{\H{o}}s.
    \newblock In {\em Combinatorial theory and its applications, {II} ({P}roc.
      {C}olloq., {B}alatonf\"{u}red, 1969)}, pages 601--623, 1970.
    
    \bibitem{KK2008}
    H.~A. Kierstead and A.~V. Kostochka.
    \newblock A short proof of the {H}ajnal-{S}zemer\'{e}di theorem on equitable
      colouring.
    \newblock {\em Combin. Probab. Comput.}, 17(2):265--270, 2008.
    
    \bibitem{KK2012}
    H.~A. Kierstead and A.~V. Kostochka.
    \newblock Every 4-colorable graph with maximum degree 4 has an equitable
      4-coloring.
    \newblock {\em J. Graph Theory}, 71(1):31--48, 2012.
    
    \bibitem{KKMS2010}
    H.~A. Kierstead, A.~V. Kostochka, M.~Mydlarz, and E.~Szemer\'{e}di.
    \newblock A fast algorithm for equitable coloring.
    \newblock {\em Combinatorica}, 30(2):217--224, 2010.
    
    \bibitem{KNP2005}
    A.~V. Kostochka, K.~Nakprasit, and S.~V. Pemmaraju.
    \newblock On equitable coloring of {$d$}-degenerate graphs.
    \newblock {\em SIAM J. Discrete Math.}, 19(1):83--95, 2005.
    
    \bibitem{Meyer1973}
    W.~Meyer.
    \newblock Equitable coloring.
    \newblock {\em Amer. Math. Monthly}, 80:920--922, 1973.
    
    \bibitem{Nakprasit20121019}
    K.~Nakprasit.
    \newblock Equitable colorings of planar graphs with maximum degree at least
      nine.
    \newblock {\em Discrete Math.}, 312(5):1019--1024, 2012.
    
    \bibitem{Thomassen1995}
    C.~Thomassen.
    \newblock Decomposing a planar graph into degenerate graphs.
    \newblock {\em J. Combin. Theory Ser. B}, 65(2):305--314, 1995.
    
    \bibitem{WVY2012}
    L.~Williams, J.~Vandenbussche, and G.~Yu.
    \newblock Equitable defective coloring of sparse planar graphs.
    \newblock {\em Discrete Math.}, 312(5):957--962, 2012.
    
    \bibitem{WZL2013}
    J.-L. Wu, X.~Zhang, and H.~Li.
    \newblock Equitable vertex arboricity of graphs.
    \newblock {\em Discrete Math.}, 313(23):2696--2701, 2013.
    
    \bibitem{Zhang2015}
    X.~Zhang.
    \newblock Equitable vertex arboricity of planar graphs.
    \newblock {\em Taiwanese J. Math.}, 19(1):123--131, 2015.
    
    \bibitem{ZhangYap1998}
    Y.~Zhang and H.-P. Yap.
    \newblock Equitable colourings of planar graphs.
    \newblock {\em Journal of Combinatorial Mathematics and Combinatorial
      Computing}, 27:97--105, 1998.
    
    \end{thebibliography}
\end{document}